\documentclass{article}

\usepackage{amsfonts}
\usepackage{amssymb}
\usepackage{eucal}
\date{}

{}

\newtheorem{theorem}{{\bf Theorem}}[section]
\newtheorem{lemma}[theorem]{{\bf Lemma}}
\newtheorem{corollary}[theorem]{{\bf Corollary}}
\newtheorem{remark}[theorem]{{\bf Remark}}

\newtheorem{definition}[theorem]{{\bf Definition}}

\newtheorem{example}[theorem]{{\bf  Example}}

\begin{document}

\begin{center}
{\Large \bf Star versions of Hurewicz spaces}\\

\medskip
{\normalsize Sumit Singh}\\
\footnotesize{Department of Mathematics,  University of Delhi, New Delhi-110007, India}\\
\footnotesize{sumitkumar405@gmail.com}\\

\medskip
{Ljubi\v sa D.R. Ko\v cinac}\\

\footnotesize {University of Ni\v s,
Faculty of Sciences and Mathematics, 18000 Ni\v s, Serbia}\\
\footnotesize{lkocinac@gmail.com }
\end{center}

\begin{abstract}
A space $X$ is said to have the set star Hurewicz property if for
each nonempty subset $A$ of $X$ and each sequence $(\mathcal{U}_n: n
\in \mathbb{N})$ of collections of sets open in $X$ such that for
each $n\in \mathbb N$, $\overline{A} \subset \cup \mathcal{U}_n$,
there is a sequence $(\mathcal{V}_n: n \in \mathbb{N})$ such that
for each $n \in \mathbb{N}$, $\mathcal{V}_n$ is a finite subset of
$\mathcal{U}_n$ and for each $x \in A$, $x \in  {\rm St}(\cup
\mathcal{V}_n, \mathcal{U}_n)$ for all but finitely many $n$. In
this paper, we investigate the relationships among set star
Hurewicz, set strongly star Hurewicz and other related covering
properties and study the topological properties of these topological
spaces.
\end{abstract}

\section{Introduction and Preliminaries}

In \cite{arh}, Arhangel'skii defined a cardinal  function $sL$, and
spaces  $X$ such that $sL(X)=\omega$ we call s-Lindel\"{o}f: a space
$X$ is \emph{s-Lindel\"{o}f} if for each subset $A$ of $X$ and each
cover $\mathcal{U}$ of $\overline{A}$ by sets open in $X$ there is a
countable set $\mathcal{V} \subset \mathcal{U}$ such that $A \subset
\overline{\cup \mathcal{V}}$. Motivated by this definition, and
modifying it,  Ko\v{c}inac and Konca \cite{koc-kon} considered new
types of selective covering properties called set covering
properties. Later on, Ko\v{c}inac, Konca and Singh in \cite{KKS}
studied set star covering properties using the star operator, and,
in particular, defined set star Hurewicz and set strongly star
Hurewicz properties.

In this paper, we investigate the relationship among set star
Hurewicz, set strongly star Hurewicz and other related properties.
Further, we study the topological properties of these two classes of
spaces.

Throughout the paper we use standard topological terminology and
notation as in \cite{engelking}. By ``a space" we mean ``a
topological space", $\mathbb{N}$ denotes the set of natural numbers,
and an open cover $\mathcal{U} $ of a subset $A \subset X$ means
elements of $\mathcal{U}$ are open in $X$ and $A \subset \cup
\mathcal{U} = \cup\{U:U \in \mathcal U\}$. The cardinality of a set
$A$ is denoted by $|A|$. Let $\omega$ denote the first infinite
cardinal, $\omega_1$ the first uncountable cardinal, $\mathfrak{c}$
the cardinality of the set of real numbers.  As usual, a cardinal is
an initial ordinal and an ordinal is the set of smaller ordinals. A
cardinal is often viewed as a space with the usual order topology.
If $A$ is a subset of a space $X$ and $\mathcal U$ is a collection
of subsets of $X$, then the \emph{star} of $A$ with respect to
$\mathcal{U}$ is the set ${\rm St}(A,\mathcal{U}) := \bigcup\{U \in
\mathcal{U}: U \cap A \neq \emptyset\}$; ${\rm St}(x,\mathcal U) =
{\rm St}(\{x\},\mathcal U)$.

We first recall the classical notions of spaces which are used in
this paper.

In 1925, Hurewicz \cite{HW,HWU} introduced the Hurewicz covering
property for a space $X$ in the following way:

A space $X$ is said to have the \emph{Hurewicz  property} if each
sequence $(\mathcal{U}_n: n \in \mathbb{N})$ of open covers of $X$
there is a sequence $(\mathcal{V}_n: n \in \mathbb{N})$ such that
for each $n \in \mathbb{N} $, $ \mathcal{V}_n$ is a finite subset of
$\mathcal{U}_n$ and for each $x \in X$, $x \in \cup \mathcal{V}_n$
for all but finitely many $n$.

Ko\v{c}inac \cite{koc-sm, koc-sm2, BCK}, introduced the star
versions of the Hurewicz covering property using the star operator
in the following way:

(1) A space $X$ is said to have the \emph{star Hurewicz property}
(shortly, $\mathsf{SH}$ property)  if each sequence $(\mathcal{U}_n:
n \in \mathbb{N})$ of  open covers of $X$, there is a sequence
$(\mathcal{V}_n: n \in \mathbb{N})$ such that for each $n \in
\mathbb{N}$, $\mathcal{V}_n$ is a finite subset of $\mathcal{U}_n$
and for each $x \in X$, $x \in  {\rm St}(\cup
\mathcal{V}_n,\mathcal{U}_n)$ for all but finitely many $n$.

(2) A space $X$ is said to have the \emph{strongly star Hurewicz
property} (shortly, $\mathsf{SSH}$ property) if for each sequence
$(\mathcal{U}_n: n \in \mathbb{N})$ of  open covers of $X$, there is
a sequence $(F_n: n \in \mathbb{N})$ of finite subsets of $X$ such
that for each $x \in X$, $x \in {\rm St}(F_n,\mathcal{U}_n)$ for all
but finitely many $n$.

In what follows we will use Theorem \ref{th1.8} below.

Recall that a collection $\mathcal{A}$ of infinite subsets of
$\omega$ is said to be \emph{almost disjoint} if the sets $A \cap B$
are finite for all distinct elements $A, B \in \mathcal{A}$. For an
almost disjoint family $\mathcal{A}$, put $\psi(\mathcal{A})=
\mathcal{A} \cup \omega$ and topologize $\psi(\mathcal{A})$  as
follows: all points in $\omega$ are isolated, and for each $A \in
\mathcal{A}$ and each finite set $F \subset \omega$, $\{A\} \cup (A
\setminus F)$ is a basic open neighborhood of $A$. The spaces of
this type are called Isbell-Mr\'{o}wka $\psi$-spaces \cite{BM,
engelking,mrowka} or $\psi (\mathcal{A})$ spaces.

\begin{theorem} {\rm (\cite{BM})} \label{th1.8}
Let $\mathcal{A}$ be an almost disjoint family of infinite subsets
of $\omega$ and let $\psi(\mathcal A) = \omega \cup \mathcal{A}$ be
the Isbell-Mr\'{o}wka space. Then:
\begin{enumerate}
\item $\psi(\mathcal A)$ is strongly star Hurewicz if and only if $|\mathcal{A}| < \mathfrak{b}$;
\item If $|\mathcal{A}| = \mathfrak{c}$, then $X$ is not star Hurewicz.
\end{enumerate}
\end{theorem}

Recently, Ko\v{c}inac and Konca \cite{koc-kon} defined the set
selection properties (and their weak versions). See also the paper
\cite{koc-add, SS} related to these properties.

In \cite{KKS}, Ko\v{c}inac, Konca and Singh defined (general
versions of) set star selection properties, in particular the set
star Hurewicz and set strongly star Hurewicz spaces.

\begin{definition}\rm
A space $X$ is said to have the
\begin{enumerate}
\item \emph{set star Hurewicz property} (shortly, {\sf set}-$\mathsf{SH}$
property) if for each nonempty set $A$ of $X$ and each sequence
$(\mathcal{U}_n: n \in \mathbb{N})$ of collections of sets open in
$X$ such that $\overline{A} \subset \cup \mathcal{U}_n$, $n \in
\mathbb{N}$, there is a sequence $(\mathcal{V}_n: n \in \mathbb{N})$
such that for each $n \in \mathbb{N}$, $\mathcal{V}_n$ is a finite
subset of $\mathcal{U}_n$ and for each $x \in A$, $x \in  {\rm St}(
\cup \mathcal{V}_n, \mathcal{U}_n)$ for all but finitely many $n$.

\item \emph{set strongly star Hurewicz property} (shortly,
{\sf set}-$\mathsf{SSH}$ property) if for each nonempty subset $A$
of $X$ and each sequence $(\mathcal{U}_n: n \in \mathbb{N})$ of
collections of sets open in $X$ such that $\overline{A} \subset \cup
\mathcal{U}_n$, $n \in \mathbb{N}$, there is a sequence $(F_n: n \in
\mathbb{N})$ of finite subsets of $\overline{A}$ such that for each
$x \in A$, $x \in {\rm St}(F_n, \mathcal{U}_n)$ for all but finitely
many $n$.
\end{enumerate}
\end{definition}

\begin{definition}\rm (\cite{douwen, matveev})
A space $X$ is said to be:
\begin{enumerate}
\item \emph{starcompact} (shortly, $\mathsf{SC}$) if for
each open cover $\mathcal{U}$ of  $X$, there is a finite subset
$\mathcal{V}$  of $\mathcal{U}$ such that $X  = {\rm St}( \cup
\mathcal{V}, \mathcal{U})$.

\item \emph{strongly starcompact}  (shortly,
$\mathsf{SSC}$) if for each open cover $\mathcal{U}$ of  $X$, there
is a finite subset  $F$  of $X$ such that $X = {\rm St}(F,
\mathcal{U})$.
\end{enumerate}
\end{definition}

In a similar way, Ko\v{c}inac, Konca and Singh \cite{KKS} considered
the following spaces.

\begin{definition}\rm
A space $X$ is said to be:
\begin{enumerate}
\item \emph{set starcompact} (shortly, {\sf set}-$\mathsf{SC}$) if for each
nonempty subset $A$ of $X$ and each open cover $\mathcal{U}$ of
$\overline{A}$, there is a finite subset $\mathcal{V}$  of
$\mathcal{U}$ such that $A = {\rm St}( \cup \mathcal{V},
\mathcal{U}) \cap A$.

\item  \emph{set strongly starcompact} (shortly, {\sf set}-$\mathsf{SSC}$) if
for each nonempty subset $A$ of $X$ and each open cover
$\mathcal{U}$ of $\overline{A}$, there is a finite subset  $F$  of
$\overline{A}$ such that $A = {\rm St}(F, \mathcal{U}) \cap A$.
\end{enumerate}
\end{definition}

It is clear, by the definition, that every set strongly starcompact
space is set starcompact.

\begin{theorem} \label{1.6} {\rm (\cite{KKS})}
Every countably compact space is set strongly starcompact.
\end{theorem}

\begin{corollary} \label{1.7} {\rm (\cite{KKS})}
Every countably compact space is set starcompact.
\end{corollary}

\section{Examples}\label{sec2}

In this section, we give some examples showing the relationships
between set star Hurewicz, set strongly star Hurewicz and other
related spaces.  Some of these examples can be found in the
literature, and we establish their additional properties.

\begin{lemma} \label{2.22}
Every Hurewicz space is set strongly star-Hurewicz.
\end{lemma}
\begin{proof}
Let $ X $ be a Hurewicz space. Let $ A $ be any nonempty subset of $
X $ and $ (\mathcal{U}_n: n \in \mathbb{N}) $ be a sequence of
collections of sets open in $ X $ such that $ \overline{A} \subset
\cup \mathcal{U}_n $. Since closed subset of Hurewicz space is
Hurewicz, thus $ \overline{A} $ is Hurewicz. Therefore there exists
a sequence $(\mathcal{V}_n: n \in \mathbb{N})$ such that for each $n
\in \mathbb{N}$, $ \mathcal{V}_n $ is a finite subset of
$\mathcal{U}_n$ and for each $x \in \overline{A}$, $x \in  \cup
\mathcal{V}_n$ for all but finitely many $n$. Choose $ x_V \in
\overline{A} \cap V $ for each $ V \in \mathcal{V}_n $. For each $ n
\in \mathbb{N} $, let $ F_n= \{x_V: V \in \mathcal{V}_n \} $. Hence
each $ F_n $ is a finite subset of $ \overline{A} $ and for each $x
\in A$, $x \in  {\rm St}(F_n, \mathcal{U}_n)$ for all but finitely
many $n$.  \end{proof}

From the definitions and the above lemma we have the following
diagram.

\[
\hskip2cm {\sf set\!-\!SSC} \rightarrow {\sf set\!-\!SC}
\]
\[
\hskip1.8cm \downarrow  \hskip2cm \downarrow
\]
\[
{\sf Hurewicz} \rightarrow {\sf set\!-\!SSH} \hskip0.2cm\rightarrow
\hskip0.2cm {\sf set\!-\!SH}
\]
\[
\hskip1.8cm \downarrow  \hskip2cm \downarrow
\]
\[
\hskip1.9cm {\sf SSH} \hskip0.5cm \rightarrow \hskip0.5cm {\sf SH}
\]
\[
{\sf Diagram \ 1}
\]

However, the converse of the implications may not be true as we show
by examples.

The following example shows that the implication {\sf Hurewicz}
$\Rightarrow$ {\sf set}-$\mathsf{SSH}$ in Diagram 1 is not
reversible.

\begin{example} \rm \label{2.1} Every countably compact non-Lindel\"{o}f space
is such an example. Such (Tychonoff) spaces are, for example, the
ordinal space $[0,\omega_1)$, the long line \cite[Example
45]{steen-seebach}, the Novak space \cite[Example
112]{steen-seebach}.
\end{example}

The following example shows that the implications {\sf
set}-$\mathsf{SSC}$  $\Rightarrow$ {\sf set}-$\mathsf{SSH}$ and {\sf
set}-$\mathsf{SC}$ $\Rightarrow$ {\sf set}-$\mathsf{SH}$ in Diagram
1 are not reversible.

\begin{example} \rm \label{2.2}
There exists a Tychonoff set strongly star Hurewicz (hence, set star
Hurewicz) space which is not set starcompact (hence, not set
strongly starcompact).
\end{example}

Indeed, let $X = D(\omega)$ be the countable discrete space. Since
$X$ is $\sigma$-compact, it is Hurewicz and thus set strongly star
Hurewicz. But $X$ is not set starcompact.

The following example shows that the implication {\sf
set}-$\mathsf{SH}$ $\Rightarrow$ $\mathsf{SH}$ in Diagram 1 is not
reversible. (Mention that the following space $X$ was considered in
several papers to obtain various counterexamples.)

\begin{example} \rm \label{2.3}
There exists a Hausdorff star Hurewicz space which is not set star
Hurewicz.
\end{example}
\begin{proof} Let $X = Y \cup A \cup\{p\}$, where
\begin{center}
$A = [0, \mathfrak{c})$, \ $B = [0,\omega)$, $Y = A \times B$,
$p\notin Y$.
\end{center}
Topologize $X$ as follows:

(i) every point of $Y$ is isolated;

(ii)  a basic neighborhood of $\alpha \in A$ is of the form
\begin{center}
$ U_{\alpha}(n)=\{\alpha\} \cup \{\langle \alpha, m \rangle : n < m
\}$
\end{center}

(iii) a basic neighborhood of $p$ takes the form
\begin{center}
$ U_(F)=\{p\} \cup  \bigcup \{\langle \alpha, n \rangle \ : \alpha
\in A \setminus F, \  n \in \omega \}$
\end{center}
for a countable subset $F$ of $A$.  In \cite[Example 2.7]{SYKH},
Song proved that $X$ is star Hurewicz.

Now we show that $ X $ is not set star Hurewicz. Consider $ A =[0,
\mathfrak{c} ) $, a closed discrete subset of $ X $. For each $
\alpha \in A $, let
\begin{center}
$ U_\alpha= \{ \alpha\} \cup \{\langle \alpha, n \rangle: n \in B \}
$.
\end{center}
Then $ U_\alpha $ is open  in $ X $ by the construction of the
topology of X and $ U_\alpha \cap U_{\alpha'} = \emptyset$ for $
\alpha \not= \alpha' $. For each $ n \in \mathbb{N} $, let $
\mathcal{U}_n= \{U_\alpha: \alpha \in A \} $. Then $ (\mathcal{U}_n:
n \in \mathbb{N}) $ is a sequence of  open covers of $ \overline{A}
$. It is enough to show that there exists a point $ \beta \in A $
such that $ \beta \notin St(\cup \mathcal{V}_n, \mathcal{U}_n) $ for
all $ n \in \mathbb{N} $, for any sequence $ (\mathcal{V}_n: n \in
\mathbb{N}) $ of finite subsets of $ \mathcal{U}_n $. Let $
(\mathcal{V}_n: n \in \mathbb{N}) $ be any sequence such that for
each $ n \in \mathbb{N} $, $ \mathcal{V}_n $ is a finite subset of $
\mathcal{U}_n $. For each $ n \in \mathbb{N} $, $ \mathcal{V}_n $ is
finite, hence there exists $ \alpha_n< \mathfrak{c} $ such that $
U_\alpha \cap (\cup \mathcal{V}_n)=\emptyset $ for each $ \alpha >
\alpha_n $. Let $ \alpha'= sup \{\alpha_n:n \in \mathbb{N}  \} $. If
we pick $ \beta > \alpha' $, then $ U_\beta \cap (\cup
\mathcal{V}_n) =\emptyset $ for each $ n \in \mathbb{N} $. Since $
U_\beta $ is the only element of $ \mathcal{U}_n $ containing the
point $ \beta $ for each $ n \in \mathbb{N} $. Thus $ \beta \notin
St(\mathcal{V}_n, \mathcal{U}_n) $ for all $ n \in \mathbb{N} $,
which shows that $ X $ is not set star Hurewicz.
\end{proof}

The following example shows that the implication {\sf
set}-$\mathsf{SSH}$  $\Rightarrow$ {\sf set}-$\mathsf{SH}$ in
Diagram 1 is not reversible.

\begin{example}\rm
There exists a Tychonoff set star Hurewicz space which is not set
strongly star Hurewicz.
\end{example}
\begin{proof}
Let $ X= A \cup B $, where $ A= \{a_\alpha: \alpha< \mathfrak{c} \}
$ is any set with $ |A|=\mathfrak{c} $ and $ B $ is any set with $
|B|=\omega $ such that any element of $ B $ is not in $ A $.
Topologize $ X $ as follows: for each $ a_\alpha \in A $ and each
finite subset $ F \subset B $, $ \{a_\alpha \} \cup (B \setminus F)
$ is a basic open neighborhood of $ a_\alpha $, and each element of
$ B $ is isolated.

First we show that $ X $ is set star Hurewicz space. Let $ C $ be
any nonempty subset of $ X $ and $ (\mathcal{U}_n:n\in \omega) $ be
any sequence of open covers of $ \overline{C} $. If $ C \subset A $,
then for each $ n \in \omega $ and for each $ a_\alpha \in C $,
there exists $ U_{\alpha,n} \in \mathcal{U}_n $ such that $ a_\alpha
\in U_{\alpha,n} $. Then for each $ a_\alpha \in C $, we can find a
finite set $ F_{\alpha,n} $ such that $ \{a_\alpha \} \cup (B
\setminus F_{\alpha,n}) \subseteq U_{\alpha,n} $. Then it is clear
that for each $ n \in \omega $ and for each $ \alpha \not= \alpha'
$, $ U_{\alpha,n} \cap U_{\alpha',n} \not= \emptyset $. For each $ n
\in \omega $, let $ \mathcal{V}_n= \{U_{\alpha,n} \} $. Then each $
\mathcal{V}_n $ is a finite subset of $ \mathcal{U}_n $ and hence
for each $a_\alpha \in C$, $a_\alpha \in {\rm St}(\cup
\mathcal{V}_n, \mathcal{U}_n)$ for all but finitely many $n$. Let $C
\subset B $. Since $B$ is countable, $ B $ is set star Hurewicz.
Hence $ X $ is set star Hurewicz.

Similarly to the proof of Example \ref{2.4}, we can prove that $X$
is not set strongly star Hurewicz.
\end{proof}

The following example shows that the implication {\sf
set}-$\mathsf{SSH}$  $\Rightarrow$ {\sf SSH} in Diagram 1
consistently is not reversible.

\begin{example}\rm  \label{2.4}
Assuming $ \omega_1< \mathfrak{b}=\mathfrak{c} $, there exists a
Tychonoff stongly star Hurewicz space $X$ which is not set strongly
star Hurewicz.
\end{example}
\begin{proof}
Let $ X=\psi(\mathcal{A})= \mathcal{A} \cup \omega $ be the
Isbell-Mr\'{o}wka space with $ |A|=\omega_1 $. Then $ X $ is
strongly star Hurewicz Tychonoff pseudocompact space (see Theorem
\ref{th1.8}).

Now we prove that $ X $ is not set strongly star Hurewicz. Let $ A=
\mathcal{A} = \{a_\alpha: \alpha< \omega_1 \}$. Then $ A $ is closed
subset of $ X $. For each $ \alpha< \omega_1 $, let $ U_\alpha=
\{a_\alpha \} \cup (a_\alpha) $. For each $ n \in \mathbb{N} $, let
$ \mathcal{U}_n= \{U_\alpha: \alpha< \omega_1 \} $. Then $
(\mathcal{U}_n: n \in \mathbb{N}) $ is a sequence of  open covers of
$ \overline{A} $. It is enough to show that there exists a point $
a_\beta  \in A $ such that
\begin{center}
$ a_\beta \notin St(F_n, \mathcal{U}_n) $ for all $  n \in
\mathbb{N} $,
\end{center}
for any sequence $ (F_n: n \in \mathbb{N}) $  of finite subsets of $
\overline{A} $. Let $ (F_n: n \in \mathbb{N}) $ be any sequence of
finite subsets of $ \overline{A} $. Then there exists $ \alpha'<
\omega_1 $ such that $ U_\alpha \cap (\bigcup_{n \in \mathbb{N} }
F_n) =\emptyset$, for each $ \alpha > \alpha' $. Pick $ \beta >
\alpha' $, then $ U_\beta \cap F_n =\emptyset$ for each $ n \in
\mathbb{N} $. Since $ U_\beta $ is the only element of $
\mathcal{U}_n $ containing the point $ a_\beta $ for each $ n \in
\mathbb{N} $. Thus $ a_\beta \notin St(F_n, \mathcal{U}_n) $ for all
$ n \in \mathbb{N} $, which shows that $ X $ is not set strongly
star Hurewicz.
\end{proof}

\begin{remark}\rm
(1) (1) In \cite{KKS}, Ko\v{c}inac et al. gave an example of
Tychonoff set starcompact space $X$ that is not set strongly
starcompact. This shows that the implication {\sf
set}-$\mathsf{SSC}$ $\Rightarrow$  {\sf set}-$\mathsf{SC}$ in
Diagram 1 is not reversible.

(2) It is known that there are star Hurewicz spaces which are not
strongly star Hurewicz (see \cite{SYHR}). This shows that the
implication $\mathsf{SSH}$  $\Rightarrow$ $\mathsf{SH}$ in  Diagram
1 is not reversible.
\end{remark}

\section{Results}

In some classes of spaces certain properties from Diagram 1
coincide. In \cite{BCK} the following theorem was proved.

\begin{theorem} {\rm (\cite[Proposition 4.1]{BCK})} \label{2.7}
If $X$ is a paracompact Hausdorff space, then $X$ is star Hurewicz
if and only if $X$ is Hurewicz.
\end{theorem}

From Theorem \ref{2.7} and Diagram 1, we have the following.

\begin{theorem} \label{2.6}
If $X$ is a paracompact Hausdorff space, then the following
statements are equivalent:
\begin{enumerate}
\item $X$ is Hurewicz;
\item $X$ is set strongly star Hurewicz;
\item $X$ is strongly star Hurewicz;
\item $X$ is set star Hurewicz;
\item $X$ is star Hurewicz.
\end{enumerate}
\end{theorem}

A space $X$ is \emph{metacompact} (resp., \emph{meta-Lindel\"{o}f})
if each open cover of $X$ has a point-finite (resp.,
point-countable) open refinement.

\begin{theorem}
Every set strongly star Hurewicz hereditarily metacompact space $X$
is a (set) Hurewicz space.
\end{theorem}
\begin{proof}
Let  $A$ be a subset of $X$ and $(\mathcal{U}_n: n \in \mathbb{N})$
be a sequence of covers of $\overline{A}$ by sets open in $X$. For
every $n \in \mathbb{N}$ the set $\bigcup \mathcal{U}_n$ is
metacompact. Let $\mathcal{V}_n$ be a point-finite open refinement
of $\mathcal{U}_n$, $n \in \mathbb{N}$. As $X$ is set strongly star
Hurewicz, there is a sequence $(F_n: n \in \mathbb{N})$ of finite
subsets of $\overline{A}$ such that for each $x \in A$, $x \in {\rm
St}(F_n, \mathcal{V}_n)$ for all but finitely many $n$. Elements of
each $F_n$ belongs to finitely many members $V_{n,1},..., V_{n,
k(n)}$ of $\mathcal{V}_n$. Let $\mathcal{W}_n= \{ V_{n,1},..., V_{n,
k(n)}\}$. Then ${\rm St}(F_n, \mathcal{V}_n)= \bigcup
\mathcal{W}_n$, so that we have for each $x \in A$, $x \in \bigcup
\mathcal{W}_n$ for all but finitely many $n$. For each $W \in
\mathcal{W}_n$ take an element $U_W$ of $ \mathcal{U}_n$ such that
$W\subset U_W$. Then, for each $n$, $\mathcal{H}_n = \{U_W: W \in
\mathcal{W}_n \} $ is a finite subset of $ \mathcal{U}_n $ and for
each $x \in A$, $x \in \bigcup \mathcal{H}_n$ for all but finitely
many $n$. It is easy to prove that this fact actually gives that $X$
is a Hurewicz space.
\end{proof}

Since every set strongly star Hurewicz space is strongly star
Hurewicz,  we have the following corollary from \cite[Theorem
2.11]{SYK} and \cite[Corollary 2.12]{SYK}.

\begin{corollary}
If $X$ is a set strongly star Hurewicz space, then the following
statements  are equivalent:
\begin{enumerate}
\item  $X$ is a meta-Lindel\"{o}f space.
\item $X$ is a para-Lindel\"{o}f space.
\item $X$ is a  Lindel\"{o}f space.
\end{enumerate}
\end{corollary}

\begin{theorem} If each nonempty set $A \subset X$ is dense in $X$ and $X$ is
star Hurewicz (resp., strongly star Hurewicz) space, then $X$ is set
star Hurewicz (resp., set strongly star Hurewicz).
\end{theorem}
\begin{proof} Because the proofs of two cases are quite similar, we prove only the
set star Hurewicz case.

Let $A$ be any nonempty subset of $X$ and $(\mathcal{U}_n: n \in
\mathbb{N})$ be a sequence of open covers of $\overline{A}$. Since
$A$ is dense in $X$, $(\mathcal{U}_n: n \in \mathbb{N})$ is a
sequence of open covers of $X$. Since $X$ is star Hurewicz, there is
a sequence $ (\mathcal{V}_n: n \in \mathbb{N}) $ such that for each
$ n \in \mathbb{N} $, $ \mathcal{V}_n $ is a finite subset of $
\mathcal{U}_n $ and for each $ x \in X$, $ x \in {\rm St}(\cup
\mathcal{V}_n, \mathcal{U}_n) $ for all but finitely many $n$. Thus
for each $ x \in A$, $ x \in  {\rm St}(\cup \mathcal{V}_n,
\mathcal{U}_n) $ for all but finitely many $ n $, which shows that $
X$ is set star Hurewicz.
\end{proof}

We now explore preservation of set star Hurewicz and set strongly
star Hurewicz spaces under basic topological constructions.

Observe that set star Hurewicz  and set strongly star Hurewicz are
not hereditary properties.  The space $X$ in Example \ref{2.4},
shows that a closed subset of a Tychonoff set star Hurewicz space
$X$ need not be set star Hurewicz. Indeed, the set $\mathcal{A}$ is
a discrete closed subset of the space $X$ in Example \ref{2.4} of
uncountable cardinality $\omega_1$, so that it cannot be set star
Hurewicz.

We saw that the ordinal space $X = [0,\omega_1)$ is set strongly
star Hurewicz. However, the subspace $Y = \{\alpha + 1: \alpha
\mbox{ is a limit ordinal }\}$ of $X$ is not set strongly star
Hurewicz.

However, we  have the following result about preservation of set
star Hurewicz and set strongly star Hurewicz spaces.

\begin{theorem} \label{3.3}
A clopen subspace of a  set star Hurewicz (resp., sert steongly star
Hureqicz) space is also set star Hurewicz (resp., set strongly star
Hurewicz).
\end{theorem}
\begin{proof} We prove only the set strongly star Hurewicz case
because the proof for set star Hurewicz case is quite similar.

Let $X$ be a set strongly star Hurewicz space and $Y \subset X$ be a
clopen subspace. Let $A$ be any subset of $Y$ and $(\mathcal{U}_n: n
\in \mathbb{N}) $ be any sequence of collections of open sets in
$(Y, \tau_Y)$ such that for each $ n \in \mathbb{N} $, $ {\rm Cl}_Y
(A) \subset \cup \mathcal{U}_n$. Since $Y$ is open, then $
(\mathcal{U}_n: n \in \mathbb{N}) $ is a sequence of collections of
open sets in $X$, and since $Y$ is closed, ${\rm Cl}_Y(A)= {\rm
Cl}_X (A) $. Applying the fact that $X$ is set strongly star
Hurewicz, we find a sequence $(F_n: n \in \mathbb{N})$ of finite
subsets of ${\rm Cl_X}(A)$ such that for each $x \in A$, $x \in {\rm
St} (F_n,\mathcal{U}_n)$ for all but finitely many $n$. Set $K_n =
F_n\cap Y$, $n \in \mathbb N$. Then the sequence $(K_n:n\in \mathbb
N)$ witnesses for $(\mathcal{U}_n: n \in \mathbb{N})$ that $Y$ is
set strongly star Hurewicz.
\end{proof}

\medskip
We now consider (non)preservation of the set star Hurewicz and set
strongly star Hurewicz properties under some sorts of mappings.

\begin{theorem} \label{3.11}
A continuous image of a set star Hurewicz space is set star
Hurewicz.
\end{theorem}
\begin{proof}
Let $X$ be a set star Hurewicz space and $f:X \rightarrow Y$ be a
continuous mapping from $X$ onto $Y$. Let $B$ be any nonempty subset
of $Y$ and   $(\mathcal{V}_n: n \in \mathbb{N})$ be a sequence of
open covers of $\overline{B}$. Let $A =f^{\gets}(B)$. Since $f$ is
continuous, for each $n \in \mathbb{N}$, $\mathcal{U}_n :=
\{f^{\gets}(V): V \in \mathcal{V}_n\}$ is the collection of open
sets in $ X $ with
\begin{center}
$ \overline{A}= \overline{f^{\gets}(B)}\subset
f^{\gets}(\overline{B}) \subset f^{\gets}(\cup \mathcal{V}_n) = \cup
\mathcal{U}_n$.
\end{center}
As $X$ is set star Hurewicz, there exists a sequence
$(\mathcal{U}_n': n \in \mathbb{N})$ such that for each $n \in
\mathbb{N}$, $\mathcal{U}_n'$ is a finite subset of $\mathcal{U}_n $
and for each $x \in A$, $x \in {\rm St}(\cup \mathcal{U}_n',
\mathcal{U}_n)$ for all but finitely many $n$. Let $\mathcal{V}_n'=
\{V:f^{\gets}(V) \in \mathcal{U}_n'\}$. Then for each $n \in
\mathbb{N}$, $\mathcal{V}_n'$ is a finite subset of $\mathcal{V}_n$.
Let $y \in B$. Then there exists $x \in A$ such that $f(x)=y$. Thus
\begin{center}
$y = f(x) \in f({\rm St}(\cup \mathcal{U}_n', \mathcal{U}_n))
\subset {\rm St}(\cup f(\{f^{\gets}(V): V \in \mathcal{V}_n'\}),
\mathcal{V}_n)= {\rm St}(\cup \mathcal{V}_n', \mathcal{V}_n)$
\end{center}
for all but finitely many $n$. Thus $Y$ is a set star Hurewicz
space.
\end{proof}

We can prove the following theorem similarly to the proof of Theorem
\ref{3.11}.

\begin{theorem}
A continuous image of a set strongly star Hurewicz space is set
strongly star Hurewicz.
\end{theorem}

Now we give a result on preimages of set strongly star Hurewicz
spaces. For this we need a new concept defined as follows. Call a
space $X$ \emph{nearly set strongly star Hurewicz} if for each $A
\subset X$ and each sequence $(\mathcal{U}_n: n \in \mathbb{N})$ of
open covers of $X$ there is a sequence $(F_n: n \in \mathbb{N})$ of
finite subsets of $X$ such that for each $x \in A$, $x \in {\rm
St}(F_n, \mathcal{U}_n)$ for all but finitely many $n$.

\begin{theorem} \label{3.14}
Let $f: X \rightarrow Y$ be an  open and closed, finite-to-one
continuous mapping from a space $X$ onto a set strongly star
Hurewicz space $Y$. Then $X$ is nearly set strongly star Hurewicz.
\end{theorem}
\begin{proof}
Let $A \subset X$ be any nonempty set and $(\mathcal{U}_n: n \in
\mathbb{N})$ be a sequence of open covers of $X$. Then $B = f(A)$ is
a nonempty subset of $Y$. Let $y \in \overline{B}$. Then
$f^{\gets}(y)$ is finite subset of $X$, and thus for each $n \in
\mathbb{N}$, there is a finite subset $\mathcal{U}_{n_y}$ of
$\mathcal{U}_n$ such that $f^{\gets}(y) \subset \bigcup
\mathcal{U}_{n_y}$ and $U \cap f^{\gets}(y)\neq \emptyset$ for each
$U \in \mathcal{U}_{n_y}$. Since $f$ is closed, there exists an open
neighborhood $V_{n_y}$ of $y$ in $Y$ such that $f^{\gets}(V_{n_y})
\subset \cup \{U:U \in \mathcal{U}_{n_y}\}$. Since $f$ is open, we
can assume that
\begin{center}
$V_{n_y} \subset \cap \{f(U):U \in \mathcal{U}_{n_y}\}$.
\end{center}
For each $n \in \mathbb{N}$, $ \mathcal{V}_n=\{V_{n_y}: y \in
\overline{B}\} $ is an open cover of $\overline{B}$. Since $Y$ is
set strongly star Hurewicz,  there exist a sequence $(F_n: n \in
\mathbb{N})$ of finite subsets of $\overline{B}$ such that for each
$y \in B$,
\begin{center}
$ y \in {\rm St}(F_n,\mathcal{V}_n)$ \, for all but finitely many
$n$.
\end{center}
Since $f$ is finite-to-one, the sequence $(f^{\gets}(F_n): n \in
\mathbb{N}) $ is the sequence of finite subsets of $X$. Now we have
to show that for each $x \in A$,
\begin{center}
$x \in  {\rm St}( f^{\gets}(F_n), \mathcal{U}_n) $ \, for all but
finitely many $n$.
\end{center}
Let $x \in A$. Then there exist $n_0 \in \mathbb{N}$ and $y \in B$
such that $y = f(x) \in V_{n,y}$ and $V_{n,y} \cap F_n \neq
\emptyset$ for all $n \geq n_0$. Since
\begin{center}
$x \in  f^{\gets}(V_{n,y}) \subset \bigcup \{U:U \in
\mathcal{U}_{n_y}\},$
\end{center}
we can choose $U \in \mathcal{U}_{n_y}$ with $x \in U$. Then
$V_{n_y} \subset f(U)$. Thus $U \cap f^{\gets}(F_n) \neq \emptyset$
for all $n \geq n_0$. Hence
\begin{center}
$x \in {\rm St}(f^{\gets} (F_n),\mathcal{U}_n) $ for all $n\geq n_0
$.
\end{center}
Thus $X$ is nearly set strongly star Hurewicz.  \end{proof}

Following the above definition of nearly set strongly star Hurewicz
spaces we will call a space X \emph{nearly set star Hurewicz} if for
each nonempty $A \subset X $ and each sequence $(\mathcal{U}_n: n
\in \mathbb{N})$ of open covers of $X$ there is a sequence
$(\mathcal{V}_n: n \in \mathbb{N})$ such that for each $n \in
\mathbb{N}$, $\mathcal{V}_n$ is a finite subset of $\mathcal{U}_n$
and for each $x \in A$, $x \in {\rm St}(\mathcal{V}_n,
\mathcal{U}_n)$ for all but finitely many $n$.

Similarly to the proof of Theorem \ref{3.14}, with necessary small
modifications, we can prove the following.

\begin{theorem} \label{3.15}
If $f: X \rightarrow Y$ is an  open perfect mapping and $Y$ is a set
star Hurewicz space, then $X$ is nearly set star Hurewicz.
\end{theorem}

From Theorem \ref{3.15} we have the following corollary.

\begin{corollary}  \label{3.16}
If $X$ is a set star Hurewicz space and $Y$ is a compact space, then
$X \times Y$ is nearly set star Hurewicz.
\end{corollary}

\begin{remark} \rm The product of two set star Hurewicz spaces need not be set
star Hurewicz. In fact, there exist two countably compact spaces $X$
and $Y$ such that $X \times Y$ is not set star Hurewicz (even not
set star Menger; see \cite{KKS}). Moreover, there exist a countably
compact (hence, set star Hurewicz) space $X$ and a Lindel\"{o}f
space $Y$ such that $X \times Y$ is not set star Hurewicz (see
\cite{KKS}).
\end{remark}

The following theorem is a version of Corollary \ref{3.16}. Call the
product $X\times Y$ \emph{rectangular set star Hurewicz} if for each
set $A\times B \subset X\times Y$ and each sequence $(\mathcal
U_n:n\in \mathbb N)$ of covers of $\overline{A\times B}$ by sets
open in $X\times Y$ there are finite sets $\mathcal V_n\subset
\mathcal U_n$, $n\in\mathbb N$, such that for each $z \in A \times
B$, $z \in {\rm St}(\cup\mathcal V_n,\mathcal U_n)$ for all but
finitely many $n$.

\begin{theorem} \label{3.17}
If $X$ is a set star Hurewicz space and $Y$ is a compact space, then
$X \times Y$ is rectangular set star Hurewicz.
\end{theorem}
\begin{proof}
Let $A= B \times C$ be any nonempty rectangular subset of $X \times
Y$ and $(\mathcal{U}_n: n \in \mathbb{N})$ be a sequence of
collections of open sets in $X \times Y$ such that $\overline{A} =
\overline{B} \times \overline{C} \subset \bigcup \mathcal{U}_n$, $n
\in \mathbb{N}$. For each $x \in \overline{B}$, $\{x\} \times
\overline{C}$ is a compact subset of $X \times Y$. Therefore, for
each $n \in \mathbb{N}$, there is a finite subset $\{U_{n,1}^x\times
V_{n,1}^x,...,U_{n,m(x)}^x\times V_{n,m(x)}^x \} $ of $
\mathcal{U}_n $ such that $\{x\} \times \overline{C} \subset
\bigcup_{i=1}^{m(x)}(U_{n,i}^x \times V_{n,i}^x)$.  For each $n \in
\mathbb{N}$, define $ W_{n}^x = \bigcap_{i=1}^{m(x)} U_{n,i}^x $.
Each $W_{n}^x$ is an open subset of $X$ containing $x$ and
\begin{center}
$ \{x\} \times \overline{ C } \subset \bigcup \{W_{n}^x \times
V_{n,i}^x: 1 \leq i \leq m(x) \}  \subset \bigcup \{U_{n,i}^x \times
V_{n,i}^x: 1 \leq i \leq m(x) \} $.
\end{center}
Then for each $n \in \mathbb{N}$, $\mathcal{W}_n= \{W_{n}^x:x \in
\overline{B}\} $ is an open cover of $\overline{B}$. Since $X$ is
set star Hurewicz, for each $n \in \mathbb{N}$, there is finite set
$ \mathcal{W}_n'= \{W_{x_j}: 1 \leq j \leq r_n\}$ of $\mathcal{W}_n$
such that for each $b \in B$,  $b \in {\rm St}(\cup\mathcal{W}_n',
\mathcal{W}_n) $ for all but finitely many $n$. For each $n \in
\mathbb{N}$, let
\begin{center}
$ \mathcal{U}_n'=\{U_{n,i}^x \times V_{n,i}^x: 1 \leq i \leq n(x_j),
1 \leq j \leq r_n\}$.
\end{center}
Then $ \mathcal{U}_n'$ is a finite subset of $\mathcal{U}_n$. Hence
for each $a \in A$,  $a \in ({\rm St}(\cup\mathcal{W}_n',
\mathcal{W}_n)\cap B) \times \overline{C} \subset {\rm
St}(\cup\mathcal{U}_n', \mathcal{U}_n)$ for all but finitely many
$n$. Thus $X \times Y $ is rectangular set star Hurewicz.
\end{proof}

\end{document}